\pdfoutput=1
\documentclass{amsart}
\usepackage{enumitem}
\usepackage[T1]{fontenc}
\usepackage[utf8]{inputenc}
\usepackage[english]{babel}
\usepackage{csquotes}
\usepackage{amsmath,amsthm,amssymb,mathtools}
\usepackage{tikz-cd}
\usepackage{microtype}
\usepackage[style=authoryear-comp,maxnames=3,backend=bibtex,
  useprefix=true,hyperref=true]{biblatex}
\usepackage[hidelinks,bookmarksdepth=2]{hyperref}
\definecolor{darkgreen}{rgb}{0,0.45,0}
\hypersetup{colorlinks,urlcolor=blue,citecolor=darkgreen,linkcolor=darkgreen,linktocpage}
\usepackage[capitalise]{cleveref}

\newtheorem{theorem}{Theorem}[section]
\newtheorem{proposition}[theorem]{Proposition}
\newtheorem{corollary}[theorem]{Corollary}
\newtheorem{lemma}[theorem]{Lemma}
\theoremstyle{definition}
\newtheorem{definition}[theorem]{Definition}

\crefformat{equation}{(#2#1#3)}

\newcommand{\htpy}{\sim}
\newcommand{\const}{\mathrm{const}}
\newcommand{\trunc}[2]{\| #2\|_{#1}}

\newcommand{\fib}{\mathrm{fib}}
\newcommand{\ct}{
  \mathchoice%
  {\mathbin{\raisebox{0.5ex}{$\displaystyle\centerdot$}}}%
  {\mathbin{\raisebox{0.5ex}{$\centerdot$}}}%
  {\mathbin{\raisebox{0.25ex}{$\scriptstyle\,\centerdot\,$}}}%
  {\mathbin{\raisebox{0.1ex}{$\scriptscriptstyle\,\centerdot\,$}}}}
\newcommand{\ap}{\mathrm{ap}}

\newcommand{\prd}[1]{\prod\nolimits_{(#1)}}
\newcommand{\lam}[1]{\lambda\,#1.\,}
\newcommand{\sm}[1]{\sum\nolimits_{(#1)}}

\newcommand{\defeq}{:\equiv}
\newcommand{\UU}{\mathcal{U}}
\newcommand{\pointedprd}[1]{\prod\nolimits^\ast_{(#1)}}
\newcommand{\unit}{\mathbf{1}}
\newcommand{\total}{\mathrm{total}}
\newcommand{\Decat}{\mathrm{Decat}}
\newcommand{\loopspace}{\Omega}
\newcommand{\gap}{\mathrm{gap}}
\newcommand{\blank}{\mathord{\hspace{1pt}\text{--}\hspace{1pt}}}
\newcommand{\im}{\mathrm{im}}
\newcommand{\pin}{\pi^{(n)}}
\newcommand{\pinf}{\pi^{(\infty)}}
\newcommand{\defemph}[1]{\textbf{#1}}

\title{The long exact sequence of homotopy $n$-groups}
\author{Ulrik Buchholtz}
\address{Technische Universit{\"a}t Darmstadt, Fachbereich Mathematik,
  Schlossgartenstra\ss e~7, 64289 Darmstadt, Germany}
\email{buchholtz@mathematik.tu-darmstadt.de}
\author{Egbert Rijke}
\address{University of Ljubljana, Fakulteta za matematiko in fiziko,
  Jadranska 19, 1000 Ljubljana, Slovenia}
\email{egbert.rijke@fmf.uni-lj.si}
\date{\today}

\begin{document}

\begin{abstract}
  Working in homotopy type theory,
  we introduce the notion of $n$-exactness for a short sequence $F\to E\to B$ of pointed types, and show that any fiber sequence $F\hookrightarrow E \twoheadrightarrow B$ of arbitrary types induces a short sequence
  \begin{equation*}
    \begin{tikzcd}
      \trunc{n-1}{F} \arrow[r] & \trunc{n-1}{E} \arrow[r] & \trunc{n-1}{B}
    \end{tikzcd}
  \end{equation*}
  that is $n$-exact at $\trunc{n-1}{E}$.
  We explain how the indexing makes sense when
  interpreted in terms of $n$-groups, and
  we compare our definition to the existing definitions of an exact sequence of $n$-groups for $n=1,2$.
  As the main application,
  we obtain the long $n$-exact sequence of homotopy $n$-groups of a fiber sequence.
\end{abstract}

\maketitle

\tableofcontents

\section{Introduction}

Homotopy type theory~\parencite{hottbook} is not only a foundational system (\emph{univalent foundations});
it also allows us to reason synthetically about $\infty$-groupoids (\emph{synthetic homotopy theory}). By viewing higher groups in terms of certain pointed $\infty$-groupoids as laid out by~\cite{BDR}, it also allows us to do \emph{synthetic higher group theory}.

From this point of view, a $1$-group $G$ is (represented by) a pointed connected $1$-type $BG$ (its classifying type). Loosely speaking, these are types that only have an interesting fundamental group, and no non-trivial higher homotopy groups.
Of course, it is not quite as simple if there are non-contractible $\infty$-connected types around, as can happen if Whitehead's principle fails.
Recall that homotopy type theory has models in $(\infty,1)$-toposes~\parencite{shulman2019},
and there are plenty such where Whitehead's principle fails.\footnote{An $(\infty,1)$-topos satisfying Whitehead's principle is also called \emph{hypercomplete}. Examples of non-hypercomplete $(\infty,1)$-toposes (in a classical metatheory)
include the $(\infty,1)$-topos of parametrized spectra (an object is hypercomplete if and only if the spectrum part is trivial) and the $(\infty,1)$-topos of continuous $\mathbb{Z}_p$-equivariant spaces, where we view the group $\mathbb{Z}_p$ of $p$-adic integers as a profinite group~\parencite[Warning~7.2.2.31]{LurieHTT}. The latter example is even boolean, hence satisfies the law of excluded middle internally.}
The underlying type of a $1$-group is therefore a set equipped with the usual structure of a group, so a group in the traditional sense of the word is a $1$-group.

Likewise, an $n$-group $G$ is represented by a connected $n$-type $BG$.
The principal example of an $n$-group is the fundamental $n$-group of a pointed type $X$, represented by the $n$-truncation of the connected component at the base point.

Many $n$-groups $G$ have further structure because they come with further deloopings of $BG$. The higher homotopy $n$-groups, $\pin_k(X)$, of a pointed type $X$ are examples of such $n$-groups with additional symmetries.
These capture the structure of $X$ in dimensions $k$ to $n+k-1$, inclusive, just like the usual higher homotopy $1$-groups, $\pi_k(X)$, capture the structure of $X$ at dimension $k$.
So whereas the usual homotopy groups discard any interactions between
different dimensions, the homotopy $n$-groups for $n>1$ retain some of that information, while still being more algebraically tractable than $X$ itself.

Our main result in this paper is \cref{maintheorem}, where we show that any fiber sequence $F\hookrightarrow E \twoheadrightarrow B$ induces a long exact sequence of homotopy $n$-groups.
The basic observation that enables this result is \cref{mainproposition}, in which we establish that the $n$-truncation operation -- although it is not left exact -- preserves $k$-cartesian squares for any $k<n$. A square
\begin{equation*}
  \begin{tikzcd}
    C\arrow[d] \arrow[r] & B \arrow[d] \\
    A \arrow[r] & X
  \end{tikzcd}
\end{equation*}
is called $k$-cartesian if the gap map $C\to A\times_X B$ is $k$-connected. In particular, any pullback square is $(n-1)$-cartesian, so the $n$-truncation of a pullback square is an $(n-1)$-cartesian square.

We work in homotopy type theory with a predicative hierarchy of
univalent universes closed under $n$-truncations.  Although we recall
the basic definitions, we refer to \parencite[Sec.~7]{hottbook}
for some results about $n$-types and the $n$-truncation modality, and
we also assume some familiarity with the basic theory of $k$-symmetric
$n$-groups as developed in~\cite{BDR}.\footnote{%
  The terminology is a bit in flux: In \emph{loc.cit.} the term was
  ``$k$-tuply groupal $(n-1)$-types, which is more in line with the
  classical notion of group-like $\mathbb{E}_k$-algebra in
  $(n-1)$-types, where $\mathbb{E}_k$ is the little $k$-cubes
  $\infty$-operad. Another proposed term is ``$(k-1)$-commutative
  $n$-group''.}

\subsection{Outline}
We start by establishing some basic definitions and notation in \cref{sec:notation}.
In \cref{sec:infty-exact} we define the notion of $\infty$-exactness and show that any fiber sequence induces a long $\infty$-exact sequence of homotopy $\infty$-groups.
In \cref{sec:kn-exact} we turn to $n$-exactness of $k$-symmetric $n$-groups and show that it is equivalent to $n$-exactness of the map on underlying $(n-1)$-types.
Our main results are in \cref{sec:mainresults}, and in \cref{sec:related} we point to some related work in the classical setting.

\section{Basic definitions and notation}\label{sec:notation}

Just as in \parencite{hottbook}, we write $x=y$ for the type of identifications of $x$ and $y$, provided that both $x$ and $y$ have a common type $X$. Sometimes we call identifications paths. We write
\begin{equation*}
  \ap_f:(x=y)\to (f(x)=f(y))
\end{equation*}
for the action of a function $f$ on paths. Path concatenation is written in diagrammatic order, i.e., we write $p \ct q$ for the concatenation of $p:x=y$ and $q:y=z$. The fiber of a map $f:A\to B$ at $b:B$ is defined to be the type
\begin{equation*}
  \fib_f(b)\defeq \sm{x:A}f(x)=b.
\end{equation*}
If $B$ is a pointed type with base point $y_0$,
we define the \defemph{kernel} of $f$ as the fiber of $f$ at $y_0$,
$\ker(f) \defeq \fib_f(y_0)$.
\begin{definition}
  A map $f:X\to Y$ is said to be an \defemph{$n$-truncation} if $Y$ is $n$-truncated, and for any family $P$ of $n$-truncated types over $Y$, the precomposition map
  \begin{equation*}
    \blank\circ f:\Big(\prd{y:Y}P(y)\Big)\to\Big(\prd{x:X}P(f(x))\Big)
  \end{equation*}
  is an equivalence. We assume that every type $X$ has an $n$-truncation
\begin{equation*}
  \eta:X\to\trunc{n}{X}.
\end{equation*}
\end{definition}

\begin{definition}
  Consider a pointed type $B$ with base point $x_0$
  and a family $E:B\to\UU$ equipped with a base point $y_0:E(x_0)$
  in the fiber over $x_0$.
  The type of \defemph{pointed sections} $\pointedprd{x:B}E(x)$ is the type of pairs $(f,p)$ consisting of a section $f:\prd{x:B}E(x)$ and an identification $p:f(x_0)=y_0$.

  Given two pointed sections $(f,p),(g,q):\pointedprd{x:B}E(x)$,
  we define the type of \defemph{pointed homotopies} as
  \begin{equation*}
    f\htpy_\ast g \defeq \pointedprd{x:B}f(x)=g(x),
  \end{equation*}
  where we equip the family of identifications given by $x\mapsto (f(x)=g(x))$
  with the base point
  \begin{equation*}
    p \ct q^{-1} : f(x_0)=g(x_0)
  \end{equation*}
  in the fiber over $x_0$.
\end{definition}
In the case of a non-dependent type family, we recover the notions of pointed maps
and pointed homotopies between these.
\begin{definition}
  A \defemph{$k$-symmetric $n$-group} $G$ is a pointed $(k-1)$-connected $(n+k-1)$-type $B^kG$. Its \defemph{underlying type} is the $k$-fold loop space $\Omega^kB^kG$. A \defemph{homomorphism} $f : G \to H$ of $k$-symmetric $n$-groups is represented by a pointed map
$B^kf : B^kG \to_* B^kH$.

We call $B^kG$ the \defemph{classifying type} of $G$.
\end{definition}

\begin{definition}\label{def:conn-cover}
  The \defemph{$m$-connected cover} $X\langle m\rangle$ of a pointed type $X$ is the kernel of $\eta : X \to \trunc mX$, equivalently,
  \begin{equation*}
    X\langle m\rangle \defeq \sm{x:X}\trunc{m-1}{x_0=x}.
  \end{equation*}
  Recall that $\eta:X\to\trunc mX$ is an $m$-connected map, so that $X\langle m\rangle$ is indeed an $m$-connected type.
\end{definition}

\begin{definition}
  The \defemph{$k$'th homotopy $n$-group} of a pointed type $X$ is represented by the $(n+k-1)$-truncation of the $(k-1)$-connected cover of $X$ at the base point, i.e., it is defined via the type
  \begin{equation*}
    B^k\pin_k(X) \defeq \trunc{n+k-1}{X\langle k-1\rangle}.
  \end{equation*}
  The underlying type of $\pin_k(X)$ is equivalent to $\trunc{n-1}{\loopspace^k X}$.
\end{definition}
Thus we see that $B^k\pin_k(X)$ fits in the fiber sequence
\begin{equation*}
  \begin{tikzcd}
    B^k\pin_k(X)\arrow[r,hook] & \trunc{n+k-1}{X} \arrow[r,->>] & \trunc{k-1}{X},
  \end{tikzcd}
\end{equation*}
Note also that in the case $k=0$ we just recover the $(n-1)$-truncation of $X$.
The observation that $B^k\pin_k(X)$ is the kernel of $\trunc{n+k-1}{X}\to\trunc{k-1}{X}$ is a generalization of the well-known fiber sequence
\begin{equation*}
  \begin{tikzcd}
    K(\pi_k(X),k) \arrow[r,hook] & \trunc{k}{X} \arrow[r,->>] & \trunc{k-1}{X}
  \end{tikzcd}
\end{equation*}
in which the fiber is the $k$'th Eilenberg-Mac Lane space of the $k$'th homotopy group of $X$ \parencite{LF}.

We can also set $n\equiv\infty$ in these definitions:
\begin{definition}
  A \defemph{$k$-symmetric $\infty$-group} $G$ is a pointed $(k-1)$-connected type $B^kG$. Its \defemph{underlying type} is the $k$-fold loop space $\Omega^kB^kG$. The \defemph{$k$'th homotopy $\infty$-group} of a pointed type $X$ is represented by $(k-1)$-connected cover of $X$ at the base point
  \begin{equation*}
    B^k\pinf_k(X) \defeq X\langle k-1\rangle,
  \end{equation*}
  so the underlying type is equivalent to $\loopspace^k X$.
\end{definition}

\section{The long \texorpdfstring{$\infty$}{∞}-exact sequence of a fiber sequence}\label{sec:infty-exact}

\begin{definition}
  A \defemph{short sequence} (or \defemph{complex}) consists of pointed types $B$, $E$, and $F$ with base points $x_0:B$, $z_0:E$ and $y_0:F$, respectively, equipped with pointed maps
  \begin{equation*}
    \begin{tikzcd}
      F \arrow[r,"i"] & E \arrow[r,"p"] & B
    \end{tikzcd}
  \end{equation*}
  and a pointed homotopy $H:p\circ_\ast i \htpy_\ast \const_{x_0}$. This homotopy witnesses that the square
  \begin{equation*}
    \begin{tikzcd}[column sep=large]
      F \arrow[r,"i"] \arrow[d,swap,"\const_\ast"] & E \arrow[d,"p"] \\
      \unit \arrow[r,swap,"\const_{x_0}"] & B
    \end{tikzcd}
  \end{equation*}
  commutes. A short sequence is said to be a \defemph{fiber sequence} if the above square is a pullback square.
\end{definition}

\begin{definition}\label{def:infty-exact}
  A short sequence
  \begin{equation*}
    \begin{tikzcd}
      F \arrow[r,"i"] & E \arrow[r,"p"] & B
    \end{tikzcd}
  \end{equation*}
  is said to be \defemph{$\infty$-exact} if the family of maps
  \begin{equation*}
    \alpha : \prd{z:E} \fib_i(z)\to (p(z)=x_0)
  \end{equation*}
  given by $\alpha(z,(y,q))= \ap_p(q)^{-1} \ct H(y)$ is a family of equivalences.
\end{definition}

\begin{proposition}
  A short sequence is $\infty$-exact if and only if it is a fiber sequence.
\end{proposition}

\begin{proof}
  First we note that we have a commuting square
  \begin{equation*}
    \begin{tikzcd}[column sep=huge]
      F \arrow[r,"\gap"] \arrow[d,swap,"\simeq"] & 1\times_B E \arrow[d,"\simeq"] \\
      \sm{z:E}\fib_i(z) \arrow[r,swap,"\total(\alpha)"] & \sm{z:E}p(z)=x_0,
    \end{tikzcd}
  \end{equation*}
  where the gap map at the top sends $y:F$ to the triple $(*,i(y),H(y))$.
  The two vertical maps in this square are equivalences. Thus we see that the gap map is an equivalence if and only if $\total(\alpha)$ is an equivalence, which is the case if and only if each $\alpha_z:\fib_i(z)\to (p(z)=x_0)$ is an equivalence.
\end{proof}

The following corollary is of course a well-known fact.\footnote{It was formalized already in Voevodsky's first \emph{UniMath} formalization, \emph{Part A}, ca.~2010--11~\parencite{UniMath}.}

\begin{corollary}
  For any fiber sequence $F\hookrightarrow E \twoheadrightarrow B$ we obtain a long $\infty$-exact sequence
  \begin{equation*}
    \begin{tikzcd}
      \cdots \arrow[r] & \loopspace F \arrow[r] & \loopspace E \arrow[r] & \loopspace B \arrow[r] & F \arrow[r] & E \arrow[r] & B.
    \end{tikzcd}
  \end{equation*}
\end{corollary}

We can reinterpret this as the sequence
\begin{equation*}
  \begin{tikzcd}[column sep=tiny]
    \cdots \ar[r] & \pinf_k(F) \ar[r] & \pinf_k(E)\ar[r] & \pinf_k(B) \ar[r] & \cdots\ar[r]
    & \pinf_0(F) \ar[r] & \pinf_0(E)\ar[r] & \pinf_0(B),
  \end{tikzcd}
\end{equation*}
where the maps into a $k$-symmetric $\infty$-group are homomorphisms of $k$-symmetric $\infty$-groups (i.e., pointed maps of the classifying types).
This motivates the following definitions and subsequent observation.

\begin{definition}
  A \defemph{short sequence} (or \defemph{complex}) of $k$-symmetric $\infty$-groups consists of three $k$-symmetric $\infty$-groups $K,G,H$ and homomorphisms
  \begin{equation*}
    \begin{tikzcd}
      K \arrow[r,"\psi"] & G \arrow[r,"\varphi"] & H,
    \end{tikzcd}
  \end{equation*}
  with an identification of $\varphi\circ\psi$ with the trivial homomorphism from $K$ to $H$ \emph{as homomorphisms}. By definition, this means we have a short sequence
  \begin{equation*}
    \begin{tikzcd}
      B^kK \arrow[r,"B^k\psi"] & B^kG \arrow[r,"B^k\varphi"] & B^kH,
    \end{tikzcd}
  \end{equation*}
  of classifying types.
\end{definition}

\begin{definition}
  Given a homomorphism of $k$-symmetric $\infty$-groups $\varphi : G \to H$, we define its \defemph{kernel}, $\ker(\varphi)$, via the classifying type $B^k\ker(\varphi) \defeq \ker(B^k\varphi)\langle k-1 \rangle$,
  that is, the $(k-1)$-connected cover of the pointed kernel at the level of classifying types, $\ker(B^k\varphi)$.
\end{definition}
At the level of underlying types, we then have $\loopspace^kB^k\ker(\varphi) \simeq \loopspace^k\ker(B^k\varphi) \simeq \ker(\loopspace^kB^k\varphi)$,
where the first equivalence is an instance of the equivalence $\Omega^k(X\langle k-1\rangle \simeq \Omega^k X$ and the second follows by iterated application of the equivalence $\loopspace(\ker(f)) \simeq \ker(\loopspace(f))$ for any pointed map $f$.
That is, the underlying type of the kernel is the kernel of the map of underlying types.
\begin{definition}
  A short sequence of $k$-symmetric $\infty$-groups $K \xrightarrow{\psi}{} G \xrightarrow{\varphi}{} H$ is \defemph{$\infty$-exact} if the induced homomorphism $K \to \ker(\varphi)$, obtained as the unique lift in the commutative square
  \begin{equation*}
    \begin{tikzcd}
      \unit \ar[r] \ar[d] & B^k\ker(\varphi) \ar[d] \\
      B^kK \ar[r]\ar[ur,dashed] & B^kG
    \end{tikzcd}
  \end{equation*}
  where the left map is $(k-2)$-connected and the right map is $(k-2)$-truncated, is an equivalence.
\end{definition}

The following proposition is the higher analogue of the fact that a group homomorphism is an isomorphism if and only if its underlying map is a bijection.

\begin{proposition}
  A homomorphism of $k$-symmetric $\infty$-groups is an equivalence if and only if the map of underlying types is an equivalence.
\end{proposition}
\begin{proof}
  This follows by induction, based on the fact that a pointed map of connected types $f : X \to Y$ is an equivalence if and only if $\loopspace f : \loopspace X \to \loopspace Y$ is~\parencite[Cor.~8.8.2]{hottbook}.
\end{proof}

\begin{corollary}
  A short sequence of $k$-symmetric $\infty$-groups is $\infty$-exact if and only if the short sequence of underlying types is $\infty$-exact.
\end{corollary}

\section{Exactness of complexes of
  \texorpdfstring{$k$}{k}-symmetric
  \texorpdfstring{$n$}{n}-groups}\label{sec:kn-exact}

Now we have laid the groundwork to consider the case of $n$-groups of finite $n$.

\begin{definition}
  A \defemph{short sequence} (or \defemph{complex}) of $k$-symmetric $n$-groups is
  a short sequence of three $k$-symmetric $\infty$-groups that happen to be $n$-groups.
\end{definition}

But beware that we have a different notion of exactness in this case, cf.~\cref{def:n-exact} below.

\begin{proposition}
  Given a homomorphism of $k$-symmetric $n$-groups $\varphi : G \to H$, the kernel $\ker(\varphi)$ is again an $n$-group.
\end{proposition}
\begin{proof}
  This follows since $\ker(B^k\varphi)$ is an $(n+k-1)$-type, and taking the $(k-1)$-connected cover preserves $(n+k-1)$-types.
\end{proof}

\begin{definition}\label{def:n-exact}
  A short sequence of $k$-symmetric $n$-groups $K \xrightarrow{\psi}{} G \xrightarrow{\varphi}{} H$ is \defemph{$n$-exact} if and only if the induced map of underlying $(n-1)$-types $K \to \ker(\varphi)$ is $(n-2)$-connected.
\end{definition}

In contrast to the $\infty$-case, we also have a useful notion of image for finite $n$:
\begin{definition}
  Given a homomorphism of $k$-symmetric $n$-groups $\varphi : G \to H$, we define the \defemph{$n$-image} $\im^n(\varphi)$ via the classifying type $B^k\im^n(\varphi)$ as it appears in the $(n+k-2)$-image factorization of $B^k\varphi$~\parencite[Def.~7.6.3]{hottbook}:
  \begin{equation}\label{eq:n-im-fact}
    \begin{tikzcd}
      B^kG \ar[r] & B^k\im^n(\varphi)\ar[r] & B^kH,
    \end{tikzcd}
  \end{equation}
  viz., $B^k\im^n(\varphi)\defeq \sm{t:B^kH}\trunc{n+k-2}{\fib_{B^k\varphi}(t)}$.
\end{definition}
When $n$ is fixed and clear from the context, we shall leave it out from the notation, and just write $\im(\varphi)$ for the ($n$-)image.
We do not mention $k$ in the notation, thanks to the following.
\begin{proposition}
  Given a homomorphism of $k$-symmetric $n$-groups $\varphi : G \to H$ with $k>0$,
  we can regard $\varphi$ as a homomorphism of underlying $(k-1)$-symmetric $n$-groups.
  Then the universal property of the $(n+k-3)$-image factorization
  induces an equivalence $\Omega B^k\im^n(\varphi) \simeq B^{k-1}\im^n(\varphi)$.
\end{proposition}
\begin{proof}
  If we apply the loop space functor to \cref{eq:n-im-fact}
  we get a factorization of $\Omega B^k\varphi = B^{k-1}\varphi$ as
  an $(n+k-3)$-connected map followed by an $(n+k-3)$-truncated map.
  Thus we get the desired induced equivalence~\parencite[Thm.~7.6.6]{hottbook}.
\end{proof}
In particular, at the level of underlying $(n-1)$-types, the $n$-image
$\im^n(\varphi)$ is the usual $(n-2)$-image.  In the special case
$n=1$ of $1$-groups, we recover the usual image (i.e., $(-1)$-image)
at the level of underlying sets.\footnote{%
  This is the reason we write a superscript $n$ for the higher
  group-theoretical $n$-image: We have to subtract $2$ when we
  describe this as an $(n-2)$-image in the sense of the truncation
  modality orthogonal factorization system at the level of underlying
  $(n-1)$-types: $\im^n(\varphi) = \im_{n-2}(\varphi)$.}

\begin{proposition}
  A short sequence of $k$-symmetric $n$-groups $K \xrightarrow{\psi}{} G \xrightarrow{\varphi}{} H$ is
  $n$-exact if and only if the unique homomorphism $\im^n(\psi) \to \ker(\varphi)$ is an equivalence.
\end{proposition}
\begin{proof}
  The map of underlying types $K \to \ker(\varphi)$ is $(n-2)$-connected if and only if the map
  $B^kK \to B^k\ker(\varphi)$ is $(n+k-2)$-connected, and this happens if and only if the right map in the $(n+k-2)$-image factorization is an equivalence.
\end{proof}

\section{The long \texorpdfstring{$n$}{n}-exact sequence of fiber sequences}\label{sec:mainresults}

Our deliberations in the previous section motivate the following definition.
\begin{definition}
  A short sequence $F\xrightarrow i{} E\xrightarrow p{} B$ of pointed types
  is \defemph{$n$-exact} if for each $z:E$, the map
  \begin{equation*}
    \alpha_z : \fib_i(z)\to (p(z)=x_0)
  \end{equation*}
  as in~\cref{def:infty-exact} is $(n-2)$-connected.
\end{definition}

\begin{definition}
  A commuting square
  \begin{equation*}
    \begin{tikzcd}
      C\arrow[d] \arrow[r] & B \arrow[d] \\
      A \arrow[r] & X
    \end{tikzcd}
  \end{equation*}
  is \defemph{$k$-cartesian} if its gap map $C \to A \times_X B$ is $k$-connected.
\end{definition}

\begin{lemma}\label{mainlemma}
  Consider a short sequence $F\xrightarrow i{} E\xrightarrow p{} B$. The following are equivalent:
  \begin{enumerate}[label=\textup(\arabic{*}\textup)]
  \item The short sequence is $n$-exact.
  \item The square
    \begin{equation*}
      \begin{tikzcd}
        F \arrow[r,"i"] \arrow[d] & E \arrow[d,"p"] \\
        \unit \arrow[r] & B
      \end{tikzcd}
    \end{equation*}
    is $(n-2)$-cartesian.
  \end{enumerate}
\end{lemma}

\begin{proof}
  Recall that the fiber of $\alpha_z$ at $q:p(z)=x_0$ is equivalent to the fiber of $\total(\alpha)$ at $(z,q):\fib_p(x_0)$~\parencite[Thm.~4.7.6]{hottbook}.
  Therefore it follows immediately that each $\alpha_z$ is $(n-2)$-connected if and only if $\total(\alpha)$ is $(n-2)$-connected.
\end{proof}

We now come to the key observation:
\begin{proposition}\label{mainproposition}
  The $n$-truncation modality preserves $k$-cartesian squares for any $k<n$.
\end{proposition}

\begin{proof}
  Consider a $k$-cartesian square
  \begin{equation*}
    \begin{tikzcd}
      C \arrow[r] \arrow[d] & B \arrow[d] \\
      A \arrow[r] & X
    \end{tikzcd}
  \end{equation*}
  for some $k<n$. Our goal is to show that the square
  \begin{equation*}
    \begin{tikzcd}
      \trunc{n}{C} \arrow[r] \arrow[d] & \trunc{n}{B} \arrow[d] \\
      \trunc{n}{A} \arrow[r] & \trunc{n}{X}
    \end{tikzcd}
  \end{equation*}
  is again $k$-cartesian. To see this, consider the commuting square
  \begin{equation*}
    \begin{tikzcd}
      C \arrow[d,swap,"\eta"] \arrow[r,"\gap"] & A\times_X B \arrow[d,"{\lam{(a,b,p)}(\eta(a),\eta(b),\ap_{\eta}(p))}"] \\
      \trunc{n}{C} \arrow[r,swap,"\gap"] & \trunc{n}{A}\times_{\trunc{n}{X}}\trunc{n}{B}
    \end{tikzcd}
  \end{equation*}
  In this square, the top map is assumed to be $k$-connected. The left map is $n$-connected, so it is also $k$-connected. Recall that if, in a commuting triangle
  \begin{equation*}
    \begin{tikzcd}[column sep=small]
      \phantom{.} \arrow[rr] \arrow[dr] & & \phantom{.} \arrow[dl] \\
      & \phantom{.}
    \end{tikzcd}
  \end{equation*}
  the top map is $k$-connected, then the left map is $k$-connected if and only if the right map is~\parencite[Lem.~1.33]{modalities}. Therefore, it suffices to show that the right map in the above square is $k$-connected. This is indeed the case, since it is the induced map on total spaces of the two $n$-connected maps $\eta:A\to\trunc{n}{A}$ and $\eta:B\to\trunc{n}{B}$, and the $(n-1)$-connected map $\ap_\eta: (f(a)=g(b))\to (\eta(f(a))=\eta(g(b)))$, all of which are also $k$-connected.
\end{proof}

Our main theorem is now a simple consequence of the above results.

\begin{theorem}\label{maintheorem}
  Any fiber sequence $F\hookrightarrow E \twoheadrightarrow B$ induces an $n$-exact short sequence $\trunc{n-1}{F}\to\trunc{n-1}{E}\to\trunc{n-1}{B}$.
\end{theorem}

\begin{proof}
  Consider a fiber sequence $F\hookrightarrow E\twoheadrightarrow B$. Since any pullback square is in particular $(n-2)$-cartesian, it follows from \cref{mainproposition} that the square
  \begin{equation*}
    \begin{tikzcd}
      \trunc{n-1}{F} \arrow[r] \arrow[d] & \trunc{n-1}{E} \arrow[d] \\
      \unit \arrow[r] & \trunc{n-1}{B}
    \end{tikzcd}
  \end{equation*}
  is $(n-2)$-cartesian. By \cref{mainlemma} it now follows that the short sequence $\trunc{n}{F}\hookrightarrow\trunc{n}{E}\twoheadrightarrow\trunc{n}{B}$ is $n$-exact.
\end{proof}

As a corollary we obtain the long $n$-exact sequence of homotopy $n$-groups, obtained from a fiber sequence $F\hookrightarrow E\twoheadrightarrow B$.

\begin{corollary}
  For any fiber sequence $F\hookrightarrow E \twoheadrightarrow B$ we obtain a long $n$-exact sequence
  \begin{equation*}
    \begin{tikzcd}
      \cdots \ar[r] &
      \pin_k(E)
      \ar[r] \ar[d, phantom, ""{coordinate, name=Z}] &
      \pin_k(B) \ar[dll, rounded corners,
      to path={ -- ([xshift=8.5ex]\tikztostart.center)
        |- (Z) [near end]\tikztonodes
        -| ([xshift=-13ex]\tikztotarget.center) -- (\tikztotarget)}] \\
      \pin_{k-1}(F) \ar[r] &
      \pin_{k-1}(E) \ar[r] \ar[d, phantom, ""{coordinate, name=W}] &
      \cdots \ar[dll, rounded corners,
      to path={ -- ([xshift=8.5ex]\tikztostart.center)
        |- (W) [near end]\tikztonodes
        -| ([xshift=-13ex]\tikztotarget.center) -- (\tikztotarget)}] \\
      \pin_0(F) \ar[r] &
      \pin_0(E) \ar[r] &
      \pin_0(B)
    \end{tikzcd}
  \end{equation*}
  of homotopy $n$-groups, where the morphisms are homomorphisms of $k$-symmetric $n$-groups whenever the codomain is a $k$-symmetric $n$-group.
\end{corollary}

As a further application we note:
\begin{corollary}\label{cor:loop-decat}
  Given a short $n$-exact sequence of $k$-symmetric $n$-groups $K \xrightarrow{\psi}{} G \xrightarrow{\varphi}{} H$, the resulting looped sequence $\loopspace K \to \loopspace G \to \loopspace H$ is a short $(n-1)$-exact sequence of $(k+1)$-symmetric $(n-1)$-groups, and the resulting decategorified sequence $\Decat(K) \to \Decat(G) \to \Decat(H)$ is a short $(n-1)$-exact sequence of $k$-symmetric $(n-1)$-groups.
\end{corollary}
Here, $\Decat$ maps a $k$-symmetric $n$-group $G$, represented by the pointed
$(k-1)$-connected $(n+k-1)$-type $B^kG$, to the $k$-symmetric $(n-1)$-group $\Decat(G)$, represented by $\trunc{n+k-2}{B^kG}$~\parencite[Sec.~6]{BDR}.

\section{Discussion and related work}
\label{sec:related}

The notion of $2$-exactness of a complex of $2$-groups is by now standard when described in terms of crossed complexes or gr-stacks \parencite{Vitale2002,butterfliesI}. Our \cref{cor:loop-decat} is reminiscent of the results of \parencite{KMV} in the setting of strict groupoids.

The benefits of our synthetic development are that we automatically
get the results in the case of stacks over a Grothendieck site as well
by interpretation in the corresponding $(\infty,1)$-topos,
and that our approach covers all higher groups, not just the case of $2$-groups
as presented by crossed modules.

\section{Acknowledgements}

The authors acknowledge the support of the Centre for Advanced Study (CAS)
at the Norwegian Academy of Science and Letters
in Oslo, Norway, which funded and hosted the research project Homotopy Type Theory and Univalent Foundations during the academic year 2018/19.

The authors are also grateful to the anonymous referees for constructive and helpful comments.

\printbibliography
\end{document}